\documentclass{amsart}
\usepackage{amsmath, amssymb, amsfonts, amsthm}
\usepackage{geometry}
\usepackage{float}

\theoremstyle{definition}
\newtheorem{definition}{Definition}

\newtheorem{proposition}{Proposition}
\newtheorem{theorem}{Theorem}
\newtheorem{corollary}{Corollary}
\newtheorem{conj}{Conjecture}
\newtheorem{exmp}{Example}
\newtheorem{rmk}{Remark}
\newtheorem*{A}{Acknowledgement}

\newenvironment{pf}{\proof}{\endproof}
\allowdisplaybreaks
\usepackage[backend=biber, style=alphabetic]{biblatex}
\title{On a weak form of Ennola's conjecture about certain cubic number fields}
\author{Jinwoo Choi}
\address{Department of Mathematical Sciences, Seoul National University, Gwanak-ro 1, Gwankak-gu, Seoul, South Korea 08826}
\email{snu1719626@snu.ac.kr}
\author{Dohyeong Kim}
\address{Department of Mathematical Sciences and Research Institute of Mathematics, Seoul National University, Gwanak-ro 1, Gwankak-gu, Seoul, South Korea 08826}
\email{dohyeongkim@snu.ac.kr}

\begin{document}

\begin{abstract}
We establish a weak form of Ennola's conjecture. We achieve this by showing that two main assumptions Louboutin made in his previous work hold true.
These assumptions are about Laurent polynomials over the rationals, and we prove them by using Newton identities.
\end{abstract}

\maketitle

\section{Introduction}

An algebraic number $\epsilon$ is called \textit{exceptional} if both $\epsilon$ and $\epsilon-1$ are units.
For example, let $l \geq 3$ be an integer, and consider a non-Galois totally real cubic number field $\mathbb{Q}(\epsilon_l)$, where the minimal polynomial of $\epsilon_l$ is 
\begin{equation}
X^3 +(l-1)X^2 -lX -1 \in \mathbb{Z}[X]. \label{eq:1}
\end{equation}
One can easily observe that $\epsilon_l$ and $\epsilon_l -1$ are both units. i.e., $\epsilon_l$ is an exceptional unit.

By Dirichlet's unit theorem, the unit group $\mathbb{U}_{l}$ of the ring of integers of $\mathbb{Q}(\epsilon_l)$ has rank 2.
Hence, one can now naturally ask the following: is $\{ \epsilon_l, \epsilon_l -1 \}$ a pair of fundamental units? That is, $\mathbb{U}_{l} = \langle -1, \epsilon_l, \epsilon_l -1 \rangle$?

This is still an open problem. Ennola conjectured in \cite{Enn91} that $\{ \epsilon_l, \epsilon_l-1 \}$ above is a pair of fundamental units for the maximal order $\mathcal{O}$ of $\mathbb{Q}(\epsilon_l)$.
We call this \textit{Ennola's conjecture}.
He showed that his conjecture is true for $3 \leq l \leq 500$ and that the \textit{unit index}
$$ j_l := (\mathbb{U}_l : \langle -1, \epsilon_l, \epsilon_l -1 \rangle )$$
of the group of units generated by $-1$, $\epsilon_l$, and $\epsilon_l -1$ in the group of units $\mathbb{U}_l$ is always coprime to 2,3, and 5.
He also showed that $\{ \epsilon_l, \epsilon_l -1 \}$ is a pair of fundamental units for $\mathcal{O}$ if $( \mathcal{O}: \mathbb{Z}[\epsilon_l] ) \leq l/3$ in \cite{Enn06}.

To show that Ennola's conjecture holds true is equivalent to prove $j_l = 1$ for $\forall l \geq 3$.
Louboutin obtained several results on this conjecture.
For example, in \cite{Lou17}, he showed that for $l \geq 3$, the unit index $j_l$ is coprime to $19!$, and $j_l = 1$ for $3 \leq l \leq 5 \cdot 10^7$.
He also proved in \cite{Lou21} that if we assume ABC conjecture is true, then Ennola's conjecture is true except for finitely many $l$. i.e., $j_l = 1$ for any sufficiently large $l$.
Another significant result in \cite{Lou21} is a conditional proof of a weak form of Ennola's conjecture that for any given integer $N \geq 2$, we have $\gcd (j_l, N!) = 1$ for $l \geq l_N$ effectively large enough.

Louboutin deduced a weak form of Ennola's conjecture from Conjectures 12 and 20 of \cite{Lou21}.
Even though Louboutin checked the validity of the conjectures by computation for finitely many case in \cite{Lou21}, the Conjectures 12 and 20 are still unsolved in \cite{Lou21}.

In this paper, we provide proofs of both conjectures and consequently establish Theorem \ref{thm1} below.

\begin{theorem}[weak form of Ennola's conjecture, Theorem 2 of \cite{Lou21}]
For any given prime $p \geq 3$, there are only finitely many $l \geq 3$ for which $p$ divides the unit index $j_l$. 
Hence, for any given integer $N \geq 2$ we have $\gcd (j_l, N!) = 1$ for $l \geq l_N$ effectively large enough.
\label{thm1}
\end{theorem}

Theorem \ref{thm1} is a statement about divisors of a unit index $j_l$.
A similar work was done by Louboutin and Lee for another family of cubic number fields. 
In \cite{LL}, they showed that the unit index $(\mathbb{U}_a: \langle -1, \epsilon_a, \epsilon_a' \rangle)$ of the group of units generated by $-1$, $\epsilon_a$, and $\epsilon_a'$ in the group of units $\mathbb{U}_a$ of the ring of integers of $\mathbb{Q}(\alpha)$ is coprime to 3 for $a \geq 1$ and coprime to 6 for $1 \leq a \equiv 2,3\ (\text{mod } 4)$, where $\alpha$ is a root of $X^3 -4a^2X+2 \in \mathbb{Z}[X]$, and $\epsilon_a, \epsilon_a'$ are two distinct units in the ring of integers of $\mathbb{Q}(\alpha)$.
Our result is similar to that of Louboutin and Lee, and we are concerned with a family of cubic polynomials \eqref{eq:1} with $l \geq 3$.
On the other hand, this family \eqref{eq:1} attracts interests both of number theorists and of topologists. For topological context, see \cite{KY}.

This paper is organized as follows. In $\S$ \ref{s2}, we introduce conjectures in \cite{Lou21}, which are used in the proof of Theorem \ref{thm1}. In $\S$ \ref{s3} and $\S$ \ref{s4}, we give proofs of conjectures.

\begin{A}
Dohyeong Kim was supported by the National Research Foundation of Korea (NRF)\footnote{the grants No.\,RS-2023-00301976 and No.\,2020R1C1C1A0100681913}. The work of Jinwoo Choi is supported by the Undergraduate Internship Program of College of Natural Sciences, Seoul National University.
\end{A}

\section{Louboutin's conjectures} \label{s2}

Conjectures in this section were necessary to prove Theorem \ref{thm1} in \cite{Lou21}, so we introduce them here. Before stating them, we need a definition of $P_d(X,Y)$, which is a basic building block of Conjecture 12.

\begin{definition} \label{def1}
For $d \geq 1$, we define the polynomial
\begin{align*}
P_d(X,Y) = d \sum_{\substack{k,l \geq 0 \\ 0 \leq 2k + 3l \leq d}} (-1)^{k-1} \binom{k+l}{k} \binom{d-k-2l}{k+l} \frac{X^{k}Y^{d-2k-3l}}{d-k-2l} \in \mathbb{Z}[X,Y].
\end{align*}
\end{definition}

One can directly compute $P_d(X,Y)$ for small $d \geq 1$ by finding all possible pairs of $(k,l)$, and below are some examples of $P_d(X,Y)$ obtained from Definition \ref{def1}.

\begin{exmp} \label{eg1}
(1) $d = 1$: $P_{1}(X,Y) = -Y$. \\
(2) $d = 2$: $P_{2}(X,Y) = -Y^{2}+2X$. \\
(3) $d = 3$: $P_{3}(X,Y) = -Y^{3}+3XY-3$. \\
(4) $d = 4$: $P_{4}(X,Y) = -Y^{4}+4XY^{2}-2X^{2}-4Y$.
\end{exmp}

Conjecture 14 implies Conjecture 12 by Proposition 15 of \cite{Lou21}. So, we introduce Conjecture 14 instead of Conjecture 12 in \cite{Lou21}. It is stated as follows:

\begin{conj}[Conjecture 14 of \cite{Lou21}] \label{conj1}
Let $a,b \in \mathbb{Z}$ be nonzero such that $c:=a+b \neq 0$. 
Set $$S_{a,b}(T) = \frac{1}{T^a} + \frac{1}{T^{b}} + T^{a+b}.$$
Then for $d \in \left\{ a,b,c \right\}$ and $P_d(X,Y)$ as in the Definition \ref{def1}, we have 
\begin{equation}
P_{|d|}(S_{a,b}(T),S_{a,b}(1/T)) = -S_{a,b}(1/T^{|d|}). \label{eq:2}
\end{equation}
Moreover, if $a$ is even and $b$ is odd, then with $$R_{a,b}(T) = \frac{1}{T^{a}} - \frac{(-1)^{a+b}}{T^{b}} + T^{a+b},$$ 
we have 
\begin{equation}
P_{|d|}(-R_{a,b}(T), -R_{a,b}(1/T)) = 
\begin{cases}
-S_{a,b}(1/T^{|d|}), & \text{if}\ d=a \\
 R_{a,b}(1/T^{|d|}), & \text{if}\ d \in \left\{ b,c \right\}.
 \end{cases} \label{eq:3}
\end{equation}
\end{conj}

There is other conjecture we have to prove to establish Theorem \ref{thm1}.
Conjecture 20 in \cite{Lou21} was necessary to apply Proposition 19 of \cite{Lou21} in the proof of Theorem \ref{thm1}, which is stated below.
$M_{a,b}$ in the statement, which are necessary to establish Theorem \ref{thm1}, are given in Table 1 of \cite{Lou21}. The values $M_{a,b}$ of our interest will be introduced in Table \ref{table:2} below.

\begin{proposition}[Proposition 19 of \cite{Lou21}] \label{prop1}
Let $a, b \in \mathbb{Z}$ not both equal to 0 be given. Let $m \geq 3$ be odd. Set $G_{a,b}(T) := F_{a,b}(R_{a,b}(T),R_{-a,-b}(T))$ with $R_{a,b}(T)$ as in Conjecture \ref{conj1}, where $F_{a,b}(X,Y)$ is given as Table \ref{table:1} below. Define 
$$ R_{a,b,m}(T) = R_{a,b}(T) + \frac{b-a}{m}T^{-a-m} + \frac{(-1)^{a+b}(a-2b)}{m}T^{-b-m} +\frac{b}{m}T^{a+b-m}.$$
Assume that 
$$G_{a,b,m}(T) = F_{a,b}(R_{a,b,m}(T), R_{-a,-b,m}(T)) \in \mathbb{Q}[T,T^{-1}] $$ is of negative degree. \\
Set $N_{a,b,m} = - \deg G_{a,b,m}(T) \geq 1$ and $B_{a,b,m}:= (M_{a,b}+N_{a,b,m}+1)/2$. If $B_{a,b,m} \leq m$, then the unit $\epsilon_{a,b} = (-1)^{a+b}\epsilon_l^a(\epsilon_l-1)^b$ is not a $m$-th power in $\mathbb{Q}(\epsilon_l)$ for $l \geq l_m$ effectively large.
\end{proposition}

Here is the table of $F_{a,b}(X,Y)$ for cases that will be dealt in this paper ($c := a+b \neq 0$):

\begin{table}[H]
\centering
\def\arraystretch{1.2}
\begin{tabular}{|c|c|}
\hline
Cases & $F_{a,b}(X,Y)$ \\ \hline \hline
Case 1: $a \geq 1$ odd and $b \geq 1$ odd &
$F_{a,b}(X,Y) = -P_{a}(Y,X) -P_{b}(Y,X) +P_{c}(X,Y)$ \\ \hline
Case 2: $a \geq 1$ odd and $b \geq 1$ even &
$F_{a,b}(X,Y) = -P_{a}(-Y,-X) -P_{b}(-Y,-X) +P_{c}(-X,-Y)$ \\ \hline
Case 3: $a \geq 2$ even and $b \geq 1$ odd &
$F_{a,b}(X,Y) = -P_{a}(-Y,-X) -P_{b}(-Y,-X) -P_{c}(-X,-Y)$ \\ \hline
\end{tabular}
\caption{Cases of $F_{a,b}(X,Y)$.}
\label{table:1}
\end{table}

\begin{conj}[Conjecture 20 of \cite{Lou21}] \label{conj2}
With notations as in Proposition \ref{prop1}, let $(a,b) \in \mathbb{Z}_{\neq 0} \times \mathbb{Z}_{\geq 1}$ be such that $m_{a,b} = a^2 + ab + b^2$ is odd and $m_{a,b} \geq 5$. Assume that the pair $(a,b)$ is not of the form $(-2b,b)$ with $b \geq 1$ odd, $(b,b)$ with $b \geq 1$ odd, with $b \geq 2$ even. \\
Then the assumptions in Proposition \ref{prop1} are satisfied for $m=m_{a,b}$ with $M_{a,b}$ and $N_{a,b,m_{a,b}}$ as in Table 2 of \cite{Lou21}.
Namely, $\deg G_{a,b,m_{a,b}}(T)<0$, and $B_{a,b,m_{a,b}} \leq m_{a,b}$.
\end{conj}

Here, we present $M_{a,b}$ and $N_{a,b,m_{a,b}}$ corresponding to each case in Table \ref{table:1}.

\begin{table}[H]
\centering
\def\arraystretch{1.2}
\begin{tabular}{|c|c|c|}
\hline
Cases in Table 1 & $M_{a,b}$ & $N_{a,b,m_{a,b}}$ \\ \hline \hline
Case 1 & $c \max(a,b)$ & $\min(a,b)^2$ \\ \hline
Case 2 & $c \max(a,b)$ & $\min(a,b)^2$ \\ \hline
Case 3 & $c \max(a,b)$ & $\min(a,b)^2$ \\ \hline
\end{tabular}
\caption{Cases of $M_{a,b}$ and $N_{a,b,m_{ab}}$.}
\label{table:2}
\end{table}

If $M_{a,b}$ and $N_{a,b,m_{a,b}}$ are as in Table \ref{table:2}, then $B_{a,b,m_{a,b}} \leq m_{a,b}$, since
$$B_{a,b,m_{a,b}} = \frac{a^2 +ab +b^2 +1}{2} = \frac{m_{a,b}+1}{2} \leq m_{a,b}$$
as \cite{Lou21} pointed out at p.15.

In summary, to prove Conjecture \ref{conj2}, we have to compute $N_{a,b,m_{a,b}}$ and verify $N_{a,b,m_{a,b}}= \min(a,b)^2$ in each case of Table \ref{table:1}.

\section{Proof of Conjecture 1} \label{s3}

The relationship between $P_d(X,Y)$ and Newton identities is important to achieve our objectives.
For notations of Newton identities, we follow conventions of \cite{CLO}.

\begin{theorem}[Newton identities, Exercise 14 in $\S 7.1$ of \cite{CLO}] \label{thm2}
Let $s_{k} = s_{k}(x_1,x_2,x_3) = x_1^{k} + x_2^{k} + x_3^{k}$ be $k$-th power sum of three variables $x_1,x_2,x_3$. 
Define $\sigma_{1} = x_1+x_2+x_3,\ \sigma_{2} = x_1x_2+x_2x_3+x_3x_1,\ \sigma_{3} = x_1x_2x_3$. 
If $d>3$, then 
\begin{align*}
s_{d}(x_1,x_2,x_3) = \sigma_1 s_{d-1}(x_1,x_2,x_3) - \sigma_2 s_{d-2}(x_1,x_2,x_3) + \sigma_3 s_{d-3}(x_1,x_2,x_3).
\end{align*}
Hence, we may regard $s_{d}(x_1,x_2,x_3)$ as a polynomial of $\sigma_1, \sigma_2, \sigma_3$ for each $d \geq 4$. i.e., 
\begin{equation}
s_{d}(x_1,x_2,x_3) = x_1^{d} + x_2^{d} + x_3^{d} = f_{d}(\sigma_1, \sigma_2, \sigma_3) \label{eq:4}
\end{equation}
for some polynomial $f_{d}$.
\end{theorem}

From Theorem \ref{thm2}, for $d>3$ we have 
\begin{equation}
f_{d}(\sigma_1,\sigma_2,\sigma_3) = \sigma_1f_{d-1}(\sigma_1,\sigma_2,\sigma_3) - \sigma_2f_{d-2}(\sigma_1,\sigma_2,\sigma_3) + \sigma_3 f_{d-3}(\sigma_1,\sigma_2,\sigma_3). \label{eq:5}
\end{equation}
Here are some examples of $f_d(\sigma_1,\sigma_2,\sigma_3)$:

\begin{exmp} \label{eg2}
(1) $d = 1$: $f_1(\sigma_1,\sigma_2,\sigma_3) = \sigma_1$.\\
(2) $d = 2$: $f_2(\sigma_1,\sigma_2,\sigma_3) = \sigma_{1}^{2}-2\sigma_{2}$.\\
(3) $d = 3$: $f_{3}(\sigma_1,\sigma_2,\sigma_3) = \sigma_{1}^{3} - 3\sigma_{1}\sigma_{2} + 3\sigma_{3}$.\\
(4) $d = 4$: $f_4(\sigma_1,\sigma_2,\sigma_3) = \sigma_{1}^{4}-4\sigma_{1}^{2}\sigma_{2}+2\sigma_{2}^{2}+4\sigma_{1}\sigma_{3}$.
\end{exmp}

In this setting, suppose $\sigma_3=1$. 
Then $\sigma_2$ becomes 
$$x_1x_2+x_2x_3+x_3x_1 = \frac{1}{x_1} + \frac{1}{x_2} + \frac{1}{x_3},$$ and \eqref{eq:5} becomes
\begin{equation}
f_{d}(\sigma_1,\sigma_2,1) = \sigma_1f_{d-1}(\sigma_1,\sigma_2,1) - \sigma_2f_{d-2}(\sigma_1,\sigma_2,1) + f_{d-3}(\sigma_1,\sigma_2,1). \label{eq:6}
\end{equation}
In terms of $x_1,x_2,x_3$, \eqref{eq:6} is
\begin{align*}
x_1^d + x_2^d + x_3^d = & f_{d} \left( x_1+x_2+x_3, \frac{1}{x_1} + \frac{1}{x_2} + \frac{1}{x_3}, 1 \right) \\
= & (x_1+x_2+x_3) f_{d-1} \left( x_1+x_2+x_3, \frac{1}{x_1} + \frac{1}{x_2} + \frac{1}{x_3},1 \right) \\
& -  \left( \frac{1}{x_1} + \frac{1}{x_2} + \frac{1}{x_3} \right) f_{d-2} \left( x_1+x_2+x_3, \frac{1}{x_1} + \frac{1}{x_2} + \frac{1}{x_3},1 \right) \\
&+ f_{d-2} \left( x_1+x_2+x_3, \frac{1}{x_1} + \frac{1}{x_2} + \frac{1}{x_3},1 \right).
\end{align*}
Put $\sigma_{2}=X$ and $\sigma_{1}=Y$. Then \eqref{eq:6} is
\begin{equation}
f_{d}(Y,X,1) = Yf_{d-1}(Y,X,1) - Xf_{d-2}(Y,X,1) + f_{d-3}(Y,X,1). \label{eq:7}
\end{equation}

We use \eqref{eq:7} in the proof of the following proposition, which is important in the proof of Conjecture \ref{conj1}.

\begin{proposition} \label{prop2}
For any $d \geq 1$, $P_{d}(X,Y) = -f_{d}(Y,X,1)$.
\end{proposition}
\begin{proof}
From Example \ref{eg1} and \ref{eg2},
\begin{align*}
-f_{1}(Y,X,1) & = -Y = P_{1}(X,Y),
\\
-f_{2}(Y,X,1) & = -Y^{2} +2X = P_{2}(X,Y),
\\
-f_{3}(Y,X,1) & = -Y^{3} +3XY -3 = P_{3}(X,Y).
\end{align*}
Now we use induction on $d$. By \eqref{eq:7}, for $d>3$ we have 
\begin{equation}
-f_{d}(Y,X,1) = Y(-f_{d-1}(Y,X,1)) - X(-f_{d-2}(Y,X,1)) + (-f_{d-3}(Y,X,1)), \label{eq:8}
\end{equation}
and by inductive hypothesis,
\begin{equation}
-f_{d-i}(Y,X,1) = (d-i) \sum_{\substack{k,l \geq 0 \\ 0 \leq 2k + 3l \leq d-i}} (-1)^{k-1} \binom{k+l}{k} \binom{d-i-k-2l}{k+l} \frac{X^{k}Y^{d-i-2k-3l}}{d-i-k-2l} \label{eq:9}
\end{equation}
for $1 \leq i \leq d-1$. Combining \eqref{eq:8} with \eqref{eq:9}, we get
\begin{align*}
-f_{d}(Y,X,1) = & Y(d-1) \sum_{\substack{k,l \geq 0 \\ 0 \leq 2k + 3l \leq d-1}} (-1)^{k-1} \binom{k+l}{k} \binom{d-1-k-2l}{k+l} \frac{X^{k}Y^{d-1-2k-3l}}{d-1-k-2l} \\
& -X (d-2) \sum_{\substack{k,l \geq 0 \\ 0 \leq 2k + 3l \leq d-2}} (-1)^{k-1} \binom{k+l}{k} \binom{d-2-k-2l}{k+l} \frac{X^{k}Y^{d-2-2k-3l}}{d-2-k-2l} \\
& + (d-3) \sum_{\substack{k,l \geq 0 \\ 0 \leq 2k + 3l \leq d-3}} (-1)^{k-1} \binom{k+l}{k} \binom{d-3-k-2l}{k+l} \frac{X^{k}Y^{d-3-2k-3l}}{d-3-k-2l} \\
\end{align*}
We explicitly compute the right hand side of the above equality. Note that $$ \binom{k+l}{k} \binom{d-k-2l}{k+l} \frac{(-1)^{k-1}d}{d-k-2l} = \frac{(-1)^{k-1}d(d-1-k-2l)!}{k!\ l!\ (d-2k-3l)!}.$$ So \eqref{eq:8} becomes
\begin{align*}
-f_{d}(Y,X,1) =& \sum_{\substack{k,l \geq 0 \\ 0 \leq 2k + 3l \leq d-1}} \frac{(-1)^{k-1}(d-1)(d-2-k-2l)!}{k!\ l!\ (d-1-2k-3l)!} X^{k}Y^{d-2k-3l} \\
&+ \sum_{\substack{k,l \geq 0 \\ 0 \leq 2k + 3l \leq d-2}} \frac{(-1)^{k}(d-2)(d-3-k-2l)!}{k!\ l!\ (d-2-2k-3l)!} X^{k+1}Y^{d-2-2k-3l} \\
&+ \sum_{\substack{k,l \geq 0 \\ 0 \leq 2k + 3l \leq d-3}} \frac{(-1)^{k-1}(d-3)(d-4-k-2l)!}{k!\ l!\ (d-3-2k-3l)!} X^{k}Y^{d-3-2k-3l}. \\
\end{align*}
To prove Proposition \ref{prop2}, it suffices to show that the summands with monomial $X^kY^{d-2k-3l}$ add up to  
\begin{equation}
\frac{(-1)^{k-1}d(d-1-k-2l)!}{k!\ l!\ (d-2k-3l)!} X^{k}Y^{d-2k-3l} \label{eq:10}
\end{equation}
for all possible pair $(i,j)$.
Before doing so, we prove the claim by dividing it into four cases. \\
\textbf{(i)} $k, l \geq 1$. \\
We obtain \eqref{eq:10} by
\begin{align*}
(-1)^{k-1} \left( \frac{(d-1)(d-2-k-2l)!}{k!\ l!\ (d-1-2k-3l)!} + \frac{(d-2)(d-2-k-2l)!}{(k-1)!\ l!\ (d-2k-3l)!} \right. \\
\left. + \frac{(d-3)(d-2-k-2l)!}{k!\ (l-1)!\ (d-2k-3l)!} \right) X^{k}Y^{d-2k-3l},
\end{align*}
where the first term comes from $Y(-f_{d-1}(Y,X,1)),$ the second term from $-X(-f_{d-2}(Y,X,1))$ for $k \rightarrow k-1$, and the third term from $f_{d-3}(Y,X,1)$ for $l \rightarrow l-1$. 
A lengthy calculation on the coefficient of $X^kY^{d-2k-3l}$ yields
\begin{align*}
& \frac{(-1)^{k-1}(d-2-k-2l)!}{(k-1)!\ (l-1)!\ (d-1-2k-3l)!} \left( \frac{d-1}{kl} + \frac{d-2}{l(d-2k-3l)} + \frac{d-3}{k(d-2k-3l)} \right) \\
=& \frac{(-1)^{k-1}(d-2-k-2l)!}{(k-1)!\ (l-1)!\ (d-1-2k-3l)!} \cdot
\frac{(d-1)(d-2k-3l) + k(d-2) + l(d-3)}{kl(d-2k-3l)} \\
=& \frac{(-1)^{k-1}(d-2-k-2l)!}{(k-1)!\ (l-1)!\ (d-1-2k-3l)!} \cdot
\frac{d(d-2k-3l)-d+dk+dl}{kl(d-2k-3l)} \\
=& \frac{(-1)^{k-1}(d-2-k-2l)!}{(k-1)!\ (l-1)!\ (d-1-2k-3l)!} \cdot
\frac{d(d-1-k-2l)}{kl(d-2k-3l)} \\
=& (-1)^{k-1} \frac{d(d-1-k-2l)!}{k!\ l!\ (d-2k-3l)!} \\
=& \binom{k+l}{k} \binom{d-k-2l}{k+l} \frac{(-1)^{k-1}d}{d-k-2l} .
\end{align*}
\textbf{(ii)} $k=0$, $l \geq 1$. \\
When $k=0$, \eqref{eq:10} becomes
\begin{equation}
-\frac{d(d-1-2l)!}{l!\ (d-3l)!}Y^{d-3l}. \label{eq:11}
\end{equation}
Since we assumed $k=0$, the second term in the case \textbf{(i)}, which was obtained by replacing $k$ with $k-1$, cannot appear. Hence,
\begin{align*}
& - \left( \frac{(d-1)(d-2-2l)!}{l!\ (d-1-3l)!} + 
\frac{(d-3)(d-2-2l)!}{(l-1)!\ (d-3l)!} \right) Y^{d-3l} \\
= & - \frac{(d-2-2l)!}{l!\ (d-3l-1)!} \left( \frac{d-1}{l} +
\frac{d-3}{d-3l} \right) Y^{d-3l}
= -\frac{d(d-1-2l)!}{l!\ (d-3l)!}Y^{d-3l}.
\end{align*}
\textbf{(iii)} $k \geq 1$, $l=0$. \\
We compute 
\begin{equation}
(-1)^{k-1} \frac{d(d-1-k)!}{k!\ (d-1-2k)!}X^{k}Y^{d-2k} \label{eq:12}
\end{equation}
as in the case \textbf{(ii)} by putting $l=0$, and excluding the third term in the case \textbf{(i)}. That is, 
\begin{align*}
& (-1)^{k-1} \left( \frac{(d-1)(d-2-k)!}{k!\ (d-1-2k)!} + 
\frac{(d-2)(d-2-k)!}{(k-1)!\ (d-2k)!} \right) X^{k}Y^{d-2k} \\
= & (-1)^{k-1} \frac{(d-2-k)!}{(k-1)!\ (d-1-2k)!} \left(
\frac{d-1}{k} + \frac{d-2}{d-2k} \right) X^{k}Y^{d-2k}
= (-1)^{k-1} \frac{d(d-1-k)!}{k!\ (d-2k)!}X^{k}Y^{d-2k}.
\end{align*}
\textbf{(iv)} $k=l=0$. \\
This term is $-Y^d$, and clearly $-Y^d$ comes from 
\begin{align*}
& Y \cdot P_{d-1}(X,Y) \\
= & Y \cdot \left( -Y^{d-1} + 
(d-1) \sum_{\substack{k,l \geq 0 \\ 0 < 2k + 3l \leq d-1}} (-1)^{k-1} \binom{k+l}{k} \binom{d-1-k-2l}{k+l} \frac{X^{k}Y^{d-1-2k-3l}}{d-1-k-2l} \right).
\end{align*}
Summing up for all such possible pairs of $(k,l)$ as \eqref{eq:10}, \eqref{eq:11}, \eqref{eq:12}, and $-Y^d$, we have 
\begin{align*}
-f_{d}(Y,X,1) =  d \sum_{\substack{k,l \geq 0 \\ 0 \leq 2k + 3l \leq d}} (-1)^{k-1} \binom{k+l}{k} \binom{d-k-2l}{k+l} \frac{X^{k}Y^{d-2k-3l}}{d-k-2l} = P_{d}(X,Y),
\end{align*} and the proof is complete. \end{proof}

Now we are ready to prove Conjecture \ref{conj1} by using Proposition \ref{prop2}. First, a corollary:

\begin{corollary} \label{cor1}
Let $x_1,x_2,x_3 \neq 0$ be such that $x_1x_2x_3=1$, and $d \geq 1$. Then $$-\left( \frac{1}{x_1^{d}} + \frac{1}{x_2^{d}} + \frac{1}{x_3^{d}} \right) = P_{d} \left( x_1+x_2+x_3, \frac{1}{x_1} + \frac{1}{x_2} + \frac{1}{x_3} \right).$$
\end{corollary}
\begin{proof}
By Proposition \ref{prop1}, 
\begin{align*}
P_{d} \left( x_1+x_2+x_3, \frac{1}{x_1} + \frac{1}{x_2} + \frac{1}{x_3} \right) = -f_{d} \left( \frac{1}{x_1} + \frac{1}{x_2} + \frac{1}{x_3}, x_1+x_2+x_3, 1 \right),
\end{align*}
and the $f_{d}$ comes from Newton identities. Combining \eqref{eq:4} with the discussion right after Example \ref{eg2}, we get the desired result.
\end{proof}

Conjecture \ref{conj1} naturally follows from Corollary \ref{cor1}.

\begin{pf}
First, we verify \eqref{eq:2}. Put $x_1 = 1/T^{a}$, $x_2 = 1/T^{b}$, $x_3 = T^{a+b}$.  Then $S_{a,b}(T) = x_1+x_2+x_3$, $x_1,x_2,x_3 \neq 0$, and $x_1x_2x_3=1$. 
From Corollary \ref{cor1},
\begin{align*}
-S_{a,b}(1/T^{|d|}) & = - \left( T^{a|d|} + T^{b|d|} + \frac{1}{T^{(a+b)|d|}} \right) = - \left( \frac{1}{x_1^{d}} + \frac{1}{x_2^{d}} + \frac{1}{x_3^{d}} \right) \\
 & = P_{d} \left( x_1+x_2+x_3, \frac{1}{x_1} + \frac{1}{x_2} + \frac{1}{x_3} \right)
= P_{|d|} ( S_{a,b}(T), S_{a,b}(1/T) ).
\end{align*} 
The proof of \eqref{eq:3} directly follows from the same method, by replacing $x_1 \rightarrow -x_1$, $x_2 \rightarrow x_2$, and $x_3 \rightarrow -x_3$. \end{pf}

\section{Proof of Conjecture 2} \label{s4}

Before we give a proof of the cases in Table \ref{table:1}, we set
$$E_{a,b}(T) = \frac{b-a}{T^a} + \frac{(-1)^{a+b}(a-2b)}{T^{b}} +bT^{a+b}.$$
For notational convenience, we let $m_{a,b}=m$ from now on. Then
\begin{align*}
R_{a,b,m}(T) & = R_{a,b}(T) + \frac{1}{mT^{m}}E_{a,b}(T),\\
R_{-a,-b,m}(T) & = R_{-a,-b}(T) + \frac{1}{mT^{m}}E_{-a,-b}(T) = R_{a,b} \left( \frac{1}{T} \right) - \frac{1}{mT^m}E_{a,b} \left( \frac{1}{T} \right).
\end{align*}
We have another thing to point out. Louboutin gives $G_{a,b}(T)$ in Conjecture 12 of \cite{Lou21} for each case, which is defined as
$$G_{a,b}(T):= F_{a,b}(R_{a,b}(T),R_{a,b}(1/T)).$$
Now that we know the conjecture holds true from the result of $\S$ \ref{s3}, we will use this fact in obtaining $\deg G_{a,b,m}(T)$. Here are the $G_{a,b}(T)$ we will use ($c := a+b \neq 0$):

\begin{table}[H]
\centering
\def\arraystretch{1.2}
\begin{tabular}{|c|c|}
\hline
Cases & $G_{a,b}(T)$ \\ \hline \hline
Case 1 &
$G_{a,b}(T) = T^{-a^2} +T^{-b^2} -T^{-c^2} +2T^{-ab}$ \\ \hline
Case 2 &
$G_{a,b}(T) = -T^{-a^2} +T^{-b^2} +T^{-c^2} +2T^{-ab}$ \\ \hline
Case 3 &
$G_{a,b}(T) = T^{-a^2} +T^{-b^2} -T^{-c^2}$ \\ \hline
\end{tabular}
\caption{$G_{a,b}(T)$ in each case.}
\label{table:3}
\end{table}

\subsection{Case 1: $a \geq 1$ odd and $b \geq 1$ odd} \label{s4.1}

In $\S$ \ref{s4.1}, we show $N_{a,b,m} = \min(a,b)^2$ from direct computation.
By Table \ref{table:1}, we have
\begin{equation}
G_{a,b,m}(T) = F_{a,b} \left( R_{a,b}(T) + \frac{1}{mT^{m}}E_{a,b}(T)
, R_{-a,-b}(T) + \frac{1}{mT^{m}}E_{-a,-b}(T) \right). \label{eq:13}
\end{equation}
Since $F_{a,b}(X,Y) = -P_{a}(Y,X) -P_{b}(Y,X) +P_{c}(X,Y)$ for Case 1, \eqref{eq:13} becomes
\begin{align}
-P_{a} \left( R_{a,b} \left( \frac{1}{T} \right) - \frac{1}{mT^m}E_{a,b} \left( \frac{1}{T} \right), R_{a,b}(T) + \frac{1}{mT^{m}}E_{a,b}(T) \right) \label{eq:14} \\
-P_{b} \left( R_{a,b} \left( \frac{1}{T} \right) - \frac{1}{mT^m}E_{a,b} \left( \frac{1}{T} \right), R_{a,b}(T) + \frac{1}{mT^{m}}E_{a,b}(T) \right) \label{eq:15} \\
+P_{c} \left( R_{a,b}(T) + \frac{1}{mT^{m}}E_{a,b}(T), R_{a,b} \left( \frac{1}{T} \right) - \frac{1}{mT^m}E_{a,b} \left( \frac{1}{T} \right) \right). \label{eq:16}
\end{align}
We can express \eqref{eq:14} as
\begin{align*}
& -P_{a} \left( R_{a,b} \left( \frac{1}{T} \right) - \frac{1}{mT^m}E_{a,b} \left( \frac{1}{T} \right), R_{a,b}(T) + \frac{1}{mT^{m}}E_{a,b}(T) \right) \\
= & a \sum_{\substack{k,l \geq 0 \\ 0 \leq 2k + 3l \leq a}} (-1)^{k} \binom{k+l}{k} \binom{a-k-2l}{k+l} \frac{1}{a-k-2l} \sum_{\substack{0 \leq i \leq k \\ 0 \leq j \leq a-2k-3l}} A_{k,i}^{a-2k-3l,j},
\end{align*}
where
\begin{align*}
A_{k,i}^{a-2k-3l,j} = & \binom{k}{i} \left( \frac{-1}{mT^{m}}E_{a,b} \left( \frac{1}{T} \right) \right)^i \left( R_{a,b} \left( \frac{1}{T} \right) \right)^{k-i} \\
& \cdot \binom{a-2k-3l}{j} \left( \frac{1}{mT^{m}}E_{a,b}(T) \right)^j (R_{a,b}(T))^{a-2k-3l-j},
\end{align*}
which can be obtained from expanding
\begin{align*}
& \left( R_{a,b} \left( \frac{1}{T} \right) - \frac{1}{mT^m}E_{a,b} \left( \frac{1}{T} \right) \right)^k \cdot \left( R_{a,b}(T) + \frac{1}{mT^{m}}E_{a,b}(T) \right)^{a-2k-3l} \\
= & \left( \sum_{i=0}^k \binom{k}{i}
 \left( \frac{-1}{mT^m}E_{a,b} \left(  \frac{1}{T} \right) \right)^i
 \left( R_{a,b} \left( \frac{1}{T} \right) \right)^{k-i} \right) \\
 & \cdot \left( \sum_{j=0}^{a-2k-3l} \binom{a-2k-3l}{j}
 \left( \frac{1}{mT^m}E_{a,b}(T) \right)^{j}
 \left( R_{a,b}(T) \right)^{a-2k-3l-j}
 \right).
\end{align*}
Note that
$$A_{k,0}^{a-2k-3l,0} = \left( R_{a,b} \left( \frac{1}{T} \right) \right)^k (R_{a,b}(T))^{a-2k-3l}.$$
Hence, we know
\begin{align*}
& -P_{a} \left( R_{a,b} \left( \frac{1}{T} \right) - \frac{1}{mT^m}E_{a,b} \left( \frac{1}{T} \right), R_{a,b}(T) + \frac{1}{mT^{m}}E_{a,b}(T) \right) \\
= & a \sum_{\substack{k,l \geq 0 \\ 0 \leq 2k + 3l \leq a}} (-1)^{k} \binom{k+l}{k} \binom{a-k-2l}{k+l} \frac{1}{a-k-2l} \left( R_{a,b} \left( \frac{1}{T} \right) \right)^k (R_{a,b}(T))^{a-2k-3l} \\
& + a \sum_{\substack{k,l \geq 0 \\ 0 \leq 2k + 3l \leq a}} (-1)^{k} \binom{k+l}{k} \binom{a-k-2l}{k+l} \frac{1}{a-k-2l} \sum_{\substack{0 \leq i \leq k \\ 0 \leq j \leq a-2k-3l \\ (i,j) \neq (0,0)}} A_{k,i}^{a-2k-3l,j} \\
= & -P_{a} \left( R_{a,b} \left( \frac{1}{T} \right), R_{a,b}(T) \right) \\
& + a \sum_{\substack{k,l \geq 0 \\ 0 \leq 2k + 3l \leq a}} (-1)^{k} \binom{k+l}{k} \binom{a-k-2l}{k+l} \frac{1}{a-k-2l} \sum_{\substack{0 \leq i \leq k \\ 0 \leq j \leq a-2k-3l \\ (i,j) \neq (0,0)}} A_{k,i}^{a-2k-3l,j}.
\end{align*}
In the same way, we can modify \eqref{eq:15} and \eqref{eq:16} as
\begin{align*}
& -P_{b} \left( R_{a,b} \left( \frac{1}{T} \right) - \frac{1}{mT^m}E_{a,b} \left( \frac{1}{T} \right), R_{a,b}(T) + \frac{1}{mT^{m}}E_{a,b}(T) \right) \\
= & -P_{b} \left( R_{a,b} \left( \frac{1}{T} \right), R_{a,b}(T) \right) \\
& + b \sum_{\substack{k,l \geq 0 \\ 0 \leq 2k + 3l \leq b}} (-1)^{k} \binom{k+l}{k} \binom{b-k-2l}{k+l} \frac{1}{b-k-2l} \sum_{\substack{0 \leq i \leq k \\ 0 \leq j \leq b-2k-3l \\ (i,j) \neq (0,0)}} B_{k,i}^{b-2k-3l,j},
\end{align*}
and 
\begin{align*}
& P_{c} \left(R_{a,b}(T) + \frac{1}{mT^{m}}E_{a,b}(T), R_{a,b} \left( \frac{1}{T} \right) - \frac{1}{mT^m}E_{a,b} \left( \frac{1}{T} \right) \right) \\
= & P_{c} \left( R_{a,b}(T), R_{a,b} \left( \frac{1}{T} \right) \right) \\
& + c \sum_{\substack{k,l \geq 0 \\ 0 \leq 2k + 3l \leq c}} (-1)^{k-1} \binom{k+l}{k} \binom{c-k-2l}{k+l} \frac{1}{c-k-2l} \sum_{\substack{0 \leq i \leq k \\ 0 \leq j \leq c-2k-3l \\ (i,j) \neq (0,0)}} C_{k,i}^{c-2k-3l,j},
\end{align*}
where
\begin{align*}
B_{k,i}^{b-2k-3l,j} = & \binom{k}{i} \left( \frac{-1}{mT^{m}}E_{a,b} \left( \frac{1}{T} \right) \right)^i \left( R_{a,b} \left( \frac{1}{T} \right) \right)^{k-i} \\
& \cdot \binom{b-2k-3l}{j} \left( \frac{1}{mT^{m}}E_{a,b}(T) \right)^j (R_{a,b}(T))^{b-2k-3l-j},
\end{align*}
and
\begin{align*}
C_{k,i}^{c-2k-3l,j} = & \binom{k}{i} \left( \frac{1}{mT^{m}}E_{a,b}(T) \right)^i (R_{a,b}(T))^{k-i} \\
& \cdot \binom{c-2k-3l}{j} \left( \frac{-1}{mT^{m}}E_{a,b} \left( \frac{1}{T} \right) \right)^j \left( R_{a,b} \left( \frac{1}{T} \right) \right)^{c-2k-3l-j}.
\end{align*}
Finally, we can express $G_{a,b,m}(T)$ as 
\begin{align*}
& -P_{a} \left( R_{a,b} \left( \frac{1}{T} \right), R_{a,b}(T) \right)  -P_{b} \left( R_{a,b} \left( \frac{1}{T} \right), R_{a,b}(T) \right)
+P_{c} \left( R_{a,b}(T), R_{a,b} \left( \frac{1}{T} \right) \right) \\
& + a \sum_{\substack{k,l \geq 0 \\ 0 \leq 2k + 3l \leq a}} (-1)^{k} \binom{k+l}{k} \binom{a-k-2l}{k+l} \frac{1}{a-k-2l} \sum_{\substack{0 \leq i \leq k \\ 0 \leq j \leq a-2k-3l \\ (i,j) \neq (0,0)}} A_{k,i}^{a-2k-3l,j} \\
& + b \sum_{\substack{k,l \geq 0 \\ 0 \leq 2k + 3l \leq b}} (-1)^{k} \binom{k+l}{k} \binom{b-k-2l}{k+l} \frac{1}{b-k-2l} \sum_{\substack{0 \leq i \leq k \\ 0 \leq j \leq b-2k-3l \\ (i,j) \neq (0,0)}} B_{k,i}^{b-2k-3l,j} \\
& + c \sum_{\substack{k,l \geq 0 \\ 0 \leq 2k + 3l \leq c}} (-1)^{k-1} \binom{k+l}{k} \binom{c-k-2l}{k+l} \frac{1}{c-k-2l} \sum_{\substack{0 \leq i \leq k \\ 0 \leq j \leq c-2k-3l \\ (i,j) \neq (0,0)}} C_{k,i}^{c-2k-3l,j}.
\end{align*}
Note that 
\begin{align*}
& -P_{a} \left( R_{a,b} \left( \frac{1}{T} \right), R_{a,b}(T) \right)  -P_{b} \left( R_{a,b} \left( \frac{1}{T} \right), R_{a,b}(T) \right)
+P_{c} \left( R_{a,b}(T), R_{a,b} \left( \frac{1}{T} \right) \right) \\
= & F_{a,b} \left( R_{a,b}(T), R_{a,b} \left( \frac{1}{T} \right) \right)
= G_{a,b}(T) = \frac{1}{T^{a^2}} +\frac{1}{T^{b^2}} -\frac{1}{T^{c^2}} +\frac{2}{T^{ab}}.
\end{align*}

So, $G_{a,b,m}(T)$ is 
\begin{align*}
& \frac{1}{T^{a^2}} +\frac{1}{T^{b^2}} -\frac{1}{T^{c^2}} +\frac{2}{T^{ab}} \\
& + a \sum_{\substack{k,l \geq 0 \\ 0 \leq 2k + 3l \leq a}} (-1)^{k} \binom{k+l}{k} \binom{a-k-2l}{k+l} \frac{1}{a-k-2l} \sum_{\substack{0 \leq i \leq k \\ 0 \leq j \leq a-2k-3l \\ (i,j) \neq (0,0)}} A_{k,i}^{a-2k-3l,j} \\
& + b \sum_{\substack{k,l \geq 0 \\ 0 \leq 2k + 3l \leq b}} (-1)^{k} \binom{k+l}{k} \binom{b-k-2l}{k+l} \frac{1}{b-k-2l} \sum_{\substack{0 \leq i \leq k \\ 0 \leq j \leq b-2k-3l \\ (i,j) \neq (0,0)}} B_{k,i}^{b-2k-3l,j} \\
& + c \sum_{\substack{k,l \geq 0 \\ 0 \leq 2k + 3l \leq c}} (-1)^{k-1} \binom{k+l}{k} \binom{c-k-2l}{k+l} \frac{1}{c-k-2l} \sum_{\substack{0 \leq i \leq k \\ 0 \leq j \leq c-2k-3l \\ (i,j) \neq (0,0)}} C_{k,i}^{c-2k-3l,j}.
\end{align*}

Next, we give formulas of $\deg A_{k,i}^{a-2k-3l,j}, \deg B_{k,i}^{b-2k-3l,j}$, and $\deg C_{k,i}^{c-2k-3l,j}$ for all possible pairs $(i,j)$ ($\neq (0,0)$) and $(k,l)$.
Note that $\deg E_{a,b}(T) = \deg R_{a,b}(T) = a+b$, $\deg R_{a,b}(1/T) = \deg E_{a,b}(1/T) = \max(a,b)$.
Hence, for all possible pairs of $i, j, k, l$, we have
\begin{align}
\deg A_{k,i}^{a-2k-3l,j} = k \max(a,b) +(a-2k-3l)(a+b) -(i+j)(a^2+ab+b^2), \label{eq:17} \\
\deg B_{k,i}^{b-2k-3l,j} = k \max(a,b) +(b-2k-3l)(a+b) -(i+j)(a^2+ab+b^2), \label{eq:18}\\
\deg C_{k,i}^{c-2k-3l,j} = k(a+b) + (c-2k-3l)\max(a,b) -(i+j)(a^2+ab+b^2). \label{eq:19}
\end{align}
Here, we divide cases based on $\max(a,b)$.

\subsubsection{$\max(a,b)=a$} \label{s4.1.1}
If $\max(a,b)=a$, then \eqref{eq:17}, \eqref{eq:18}, and \eqref{eq:19} are
\begin{align*}
\text{The RHS of \eqref{eq:17}} & = ka +(a-2k-3l)(a+b) -(i+j)(a^2+ab+b^2) \\
& = -k(a+2b) -3(a+b)l + a(a+b) -(i+j)(a^2+ab+b^2), \\
\text{that of \eqref{eq:18}} & = ka +(b-2k-3l)(a+b) -(i+j)(a^2+ab+b^2) \\
& = -k(a+2b) -3(a+b)l +b(a+b) -(i+j)(a^2+ab+b^2), \\
\text{that of \eqref{eq:19}} & = k(a+b) + (c-2k-3l)a -(i+j)(a^2+ab+b^2) \\
& = k(-a+b) -3al +a(a+b) -(i+j)(a^2+ab+b^2).
\end{align*}
Note that $-a+b<0$ since $\max(a,b)=a$ with $a \neq b$, and $i+j \geq 1$.
Now, for $k,l \in \mathbb{Z}_{\geq 0}$, we check
\begin{align}
\max \{ \deg A_{k,i}^{a-2k-3l,j} : 0 \leq i \leq k, 0 \leq j \leq a-2k-3l, (i,j) \neq (0,0), 0 \leq 2k + 3l \leq a \}, \label{eq:20} \\
\max \{ \deg B_{k,i}^{b-2k-3l,j} : 0 \leq i \leq k, 0 \leq j \leq b-2k-3l, (i,j) \neq (0,0), 0 \leq 2k + 3l \leq b \}, \label{eq:21} \\
\max \{ \deg C_{k,i}^{c-2k-3l,j} : 0 \leq i \leq k, 0 \leq j \leq c-2k-3l, (i,j) \neq (0,0), 0 \leq 2k + 3l \leq c \}. \label{eq:22}
\end{align}
For \eqref{eq:17} to be maximum, we know $k=l=i=0,j=1$. In this case, $\eqref{eq:20} = -b^{2}$. \eqref{eq:18} and \eqref{eq:19} also obtain maximum at $k=l=i=0,j=1$, and for such $i,j,k,l$, $\eqref{eq:21} = -a^2$, and $\eqref{eq:22} = -b^2$. We do not have to consider other cases of $i,j,k,l$, because as $k,l,i$, and $j$ increase, the degree of $A_{k,i}^{a-2k-3l,j}, B_{k,i}^{b-2k-3l,j}$ and $C_{k,i}^{c-2k-3l,j}$ decreases, so they do not contribute to $\deg G_{a,b,m}(T)$.

From the expression of $G_{a,b,m}(T)$ right above the equation \eqref{eq:17}, we  have $\deg G_{a,b,m}(T) \leq -b^2 = -\min(a,b)^2$.
To make = hold, it suffices to show that the coefficient of $1/T^{b^2}$ in $G_{a,b,m}(T)$ is nonzero.

Let $\text{lc}(Q)$ be the leading coefficient of a polynomial $Q \in \mathbb{Q}[T,T^{-1}]$, and $\alpha_{k}(Q)$ be the coefficient of the monomial $T^k$ in $Q \in \mathbb{Q}[T,T^{-1}],\ k \in \mathbb{Z}$. Then $\alpha_{-b^2}(G_{a,b,m})$ comes from $\text{lc}(A_{0,0}^{a,1})$, $\text{lc}(C_{0,0}^{c,1})$, and $1/T^{b^2}$ of $G_{a,b}(T)$. Namely,
\begin{align*}
\alpha_{-b^2}(G_{a,b,m}) = a & \binom{0+0}{0} \binom{a-0-0}{0} \frac{1}{a-0-0} \cdot \text{lc}(A_{0,0}^{a,1}) \\
& + c \binom{0+0}{0} \binom{c-0-0}{0} \frac{-1}{c-0-0} \cdot \text{lc}(C_{0,0}^{c,1}) + 1. 
\end{align*}
Carefully expanding $A_{0,0}^{a,1}$ and $C_{0,0}^{c,1}$, we know
\begin{align*}
a \binom{0+0}{0} \binom{a-0-0}{0} \frac{1}{a-0-0}
\cdot \text{lc}(A_{0,0}^{a,1}) = \frac{ab}{a^2+ab+b^2}, \\
c \binom{0+0}{0} \binom{c-0-0}{0} \frac{-1}{c-0-0}
 \cdot \text{lc}(C_{0,0}^{c,1}) = \frac{b^2-a^2}{a^2+ab+b^2}.
\end{align*}
Hence,
\begin{align*}
\alpha_{-b^2}(G_{a,b,m}) = \frac{ab}{a^2+ab+b^2} +\frac{b^2-a^2}{a^2+ab+b^2} +1 = \frac{2b(a+b)}{a^2+ab+b^2} = \frac{2c \min(a,b)}{a^2+ab+b^2} \neq 0,
\end{align*}
and we obtain $\alpha_{-b^2}(G_{a,b,m}) = \text{lc}(G_{a,b,m}) \neq 0$.

\begin{rmk}
While the numerator of $\text{lc}(G_{a,b,m})$ and that of $q_{n}$ in Table 2 of \cite{Lou21} coincide, the denominators are different.
This happens in other cases below.
But our result agrees with Louboutin's prediction on $N_{a,b,m}$.
\end{rmk}

\subsubsection{$\max(a,b)=b$}
We repeat what we did in \ref{s4.1.1}.
If $\max(a,b)=b$, then \eqref{eq:17}, \eqref{eq:18}, and \eqref{eq:19} are
\begin{align*}
\text{The RHS of \eqref{eq:17}} & = kb +(a-2k-3l)(a+b) -(i+j)(a^2+ab+b^2) \\
& = -k(2a+b) -3(a+b)l + a(a+b) -(i+j)(a^2+ab+b^2), \\
\text{That of \eqref{eq:18}} & = kb +(b-2k-3l)(a+b) -(i+j)(a^2+ab+b^2) \\
& = -k(2a+b) -3(a+b)l +b(a+b) -(i+j)(a^2+ab+b^2), \\
\text{That of \eqref{eq:19}} & = k(a+b) + (c-2k-3l)b -(i+j)(a^2+ab+b^2) \\
& = k(a-b) -3bl +b(a+b) -(i+j)(a^2+ab+b^2).
\end{align*}
We also have the maximum of $\deg A_{k,i}^{a-2k-3l,j}, \deg B_{k,i}^{b-2k-3l,j}$, and $\deg C_{k,i}^{c-2k-3l,j}$ when $k=l=i=0$, $j=1$. Then $\deg A_{0,0}^{a,1}=-b^2$, $\deg B_{0,0}^{b,1}=-a^2$, and $\deg C_{0,0}^{c,1}=-a^2$. So, $\deg G_{a,b,m}(T) = -a^2$ provided $\alpha_{-a^2}(G_{a,b,m}) \neq 0$. By checking $\alpha_{-a^2}(G_{a,b,m})$ as
\begin{align*}
\alpha_{-a^2}(G_{a,b,m}) = & b \binom{0+0}{0} \binom{b-0-0}{0} \frac{1}{b-0-0} \cdot \text{lc}(B_{0,0}^{b,1}) \\
 & + c \binom{0+0}{0} \binom{c-0-0}{0} \frac{-1}{c-0-0} \cdot \text{lc}(C_{0,0}^{c,1}) + 1 \\
= & \frac{b^2}{a^2+ab+b^2} +\frac{(b+a)(a-2b)}{a^2+ab+b^2} +1
= \frac{2a^2}{a^2+ab+b^2} = \frac{2\min(a,b)a}{a^2+ab+b^2} \neq 0,
\end{align*}
we have $\alpha_{-a^2}(G_{a,b,m}) = \text{lc}(G_{a,b,m}) \neq 0$ and conclude $\deg G_{a,b,m}(T) = -a^2 = -\min(a,b)^2 = -N_{a,b,m}$.

\subsection{Case 2: $a \geq 1$ odd and $b \geq 1$ even}

The argument in this subsection is similar to $\S$ \ref{s4.1}.
For Case 2, $F_{a,b}(X,Y)$ and $G_{a,b}(T)$ are given by 
$F_{a,b}(X,Y) = -P_{a}(-Y,-X) -P_{b}(-Y,-X) +P_{c}(-X,-Y)$, and $G_{a,b}(T) = -T^{-a^2} +T^{-b^2} +T^{-c^2} +2T^{-ab}$.
As in Case 1, we put $X = R_{a,b,m}(T)$ and $Y = R_{-a,-b,m}(T)$.
Then $G_{a,b,m}(T)$ satisfies
\begin{align*}
G_{a,b,m}(T) = & -P_{a} \left( -\left( R_{a,b} \left( \frac{1}{T} \right) - \frac{1}{mT^m}E_{a,b} \left( \frac{1}{T} \right) \right),  -\left( R_{a,b}(T) + \frac{1}{mT^{m}}E_{a,b}(T) \right) \right) \\
& -P_{b} \left( -\left( R_{a,b} \left( \frac{1}{T} \right) - \frac{1}{mT^m}E_{a,b} \left( \frac{1}{T} \right) \right), -\left( R_{a,b}(T) + \frac{1}{mT^{m}}E_{a,b}(T) \right) \right) \\
& +P_{c} \left( -\left( R_{a,b}(T) + \frac{1}{mT^{m}}E_{a,b}(T) \right), -\left( R_{a,b} \left( \frac{1}{T} \right) - \frac{1}{mT^m}E_{a,b} \left( \frac{1}{T} \right) \right) \right) \\
= &  -P_{a} \left(-R_{a,b} \left( \frac{1}{T} \right), -R_{a,b}(T) \right)  -P_{b} \left( -R_{a,b} \left( \frac{1}{T} \right), -R_{a,b}(T) \right) \\
& +P_{c} \left( -R_{a,b}(T), -R_{a,b} \left( \frac{1}{T} \right) \right) \\
& + a \sum_{\substack{k,l \geq 0 \\ 0 \leq 2k + 3l \leq a}} (-1)^{k} \binom{k+l}{k} \binom{a-k-2l}{k+l} \frac{(-1)^{a-k-3l}}{a-k-2l} \sum_{\substack{0 \leq i \leq k \\ 0 \leq j \leq a-2k-3l \\ (i,j) \neq (0,0)}} A_{k,i}^{a-2k-3l,j} \\
& + b \sum_{\substack{k,l \geq 0 \\ 0 \leq 2k + 3l \leq b}} (-1)^{k} \binom{k+l}{k} \binom{b-k-2l}{k+l} \frac{(-1)^{b-k-3l}}{b-k-2l} \sum_{\substack{0 \leq i \leq k \\ 0 \leq j \leq b-2k-3l \\ (i,j) \neq (0,0)}} B_{k,i}^{b-2k-3l,j} \\
& + c \sum_{\substack{k,l \geq 0 \\ 0 \leq 2k + 3l \leq c}} (-1)^{k-1} \binom{k+l}{k} \binom{c-k-2l}{k+l} \frac{(-1)^{c-k-3l}}{c-k-2l} \sum_{\substack{0 \leq i \leq k \\ 0 \leq j \leq c-2k-3l \\ (i,j) \neq (0,0)}} C_{k,i}^{c-2k-3l,j}\\
= & -\frac{1}{T^{a^2}} +\frac{1}{T^{b^2}} +\frac{1}{T^{c^2}} +\frac{2}{T^{ab}} \\
& + a \sum_{\substack{k,l \geq 0 \\ 0 \leq 2k + 3l \leq a}} (-1)^{k} \binom{k+l}{k} \binom{a-k-2l}{k+l} \frac{(-1)^{a-k-3l}}{a-k-2l} \sum_{\substack{0 \leq i \leq k \\ 0 \leq j \leq a-2k-3l \\ (i,j) \neq (0,0)}} A_{k,i}^{a-2k-3l,j} \\
& + b \sum_{\substack{k,l \geq 0 \\ 0 \leq 2k + 3l \leq b}} (-1)^{k} \binom{k+l}{k} \binom{b-k-2l}{k+l} \frac{(-1)^{b-k-3l}}{b-k-2l} \sum_{\substack{0 \leq i \leq k \\ 0 \leq j \leq b-2k-3l \\ (i,j) \neq (0,0)}} B_{k,i}^{b-2k-3l,j} \\
& + c \sum_{\substack{k,l \geq 0 \\ 0 \leq 2k + 3l \leq c}} (-1)^{k-1} \binom{k+l}{k} \binom{c-k-2l}{k+l} \frac{(-1)^{c-k-3l}}{c-k-2l} \sum_{\substack{0 \leq i \leq k \\ 0 \leq j \leq c-2k-3l \\ (i,j) \neq (0,0)}} C_{k,i}^{c-2k-3l,j}.
\end{align*}
In summary, we have
\begin{align}
\deg A_{k,i}^{a-2k-3l,j} = k \max(a,b) +(a-2k-3l)(a+b) -(i+j)(a^2+ab+b^2), \label{eq:23} \\
\deg B_{k,i}^{b-2k-3l,j} = k \max(a,b) +(b-2k-3l)(a+b) -(i+j)(a^2+ab+b^2), \label{eq:24} \\
\deg C_{k,i}^{c-2k-3l,j} = k(a+b) + (c-2k-3l)\max(a,b) -(i+j)(a^2+ab+b^2). \label{eq:25}
\end{align}

\subsubsection{$\max(a,b)=a$} \label{s4.2.1}
We have
\begin{align*}
\text{The RHS of \eqref{eq:23}} & = ka +(a-2k-3l)(a+b) -(i+j)(a^2+ab+b^2) \\
& = -k(a+2b) -3(a+b)l + a(a+b) -(i+j)(a^2+ab+b^2), \\
\text{That of \eqref{eq:24}} & = ka +(b-2k-3l)(a+b) -(i+j)(a^2+ab+b^2) \\
& = -k(a+2b) -3(a+b)l +b(a+b) -(i+j)(a^2+ab+b^2), \\
\text{That of \eqref{eq:25}} & = k(a+b) + (c-2k-3l)a -(i+j)(a^2+ab+b^2) \\
& = k(-a+b) -3al +a(a+b) -(i+j)(a^2+ab+b^2).
\end{align*}
For \eqref{eq:23}, \eqref{eq:24} and \eqref{eq:25} to attain maximum, $k=l=i=0,j=1$. In this case, $\deg A_{0,0}^{a,1}=-b^2$, $\deg B_{0,0}^{b,1}=-a^2$, and $\deg C_{0,0}^{c,1}=-b^2$. We do not have to consider other tuples of $k,l,i,j$, since the degree is strictly smaller for $(k,l,i,j) \neq (0,0,0,1)$. Observing $\alpha_{-b^2}(G_{a,b,m})$ as
\begin{align*}
& a \binom{0+0}{0} \binom{a-0-0}{0} \frac{(-1)^a}{a-0-0} \cdot \text{lc}(A_{0,0}^{a,1})
 + c \binom{0+0}{0} \binom{c-0-0}{0} \frac{(-1)^{c+1}}{c-0-0} \cdot \text{lc}(C_{0,0}^{c,1}) + 1 \\
& = -\frac{ab}{a^2+ab+b^2} -\frac{b^2-a^2}{a^2+ab+b^2} +1 
= \frac{2a^2}{a^2+ab+b^2} \neq 0,
\end{align*}
we know $\alpha_{-b^2}(G_{a,b,m})=\text{lc}(G_{a,b,m}) \neq 0$ and $\deg G_{a,b,m}(T) = -b^2 = - \min(a,b)^2 = -N_{a,b,m}$.

\subsubsection{$\max(a,b)=b$}
We have
\begin{align*}
\text{The RHS of \eqref{eq:23}} & = kb +(a-2k-3l)(a+b) -(i+j)(a^2+ab+b^2) \\
& = -k(2a+b) -3(a+b)l + a(a+b) -(i+j)((a^2+ab+b^2), \\
\text{That of \eqref{eq:24}} & = kb +(b-2k-3l)(a+b) -(i+j)(a^2+ab+b^2) \\
& = -k(2a+b) -3(a+b)l +b(a+b) -(i+j)((a^2+ab+b^2), \\
\text{That of \eqref{eq:25}} & = k(a+b) + (c-2k-3l)b -(i+j)(a^2+ab+b^2) \\
& = k(a-b) -3bl +b(a+b) -(i+j)(a^2+ab+b^2).
\end{align*}
As in \ref{s4.2.1}, we only have to consider $k=l=i=0,j=1$. Observing $\alpha_{-a^2}(G_{a,b,m})$ by
\begin{align*}
& b \binom{0+0}{0} \binom{b-0-0}{0} \frac{(-1)^b}{b-0-0} \cdot \text{lc}(B_{0,0}^{b,1}) \\
& + c \binom{0+0}{0} \binom{c-0-0}{0} \frac{(-1)^{c+1}}{c-0-0} \cdot \text{lc}(C_{0,0}^{c,1}) -1 \\
= & \frac{b^2}{a^2+ab+b^2} +\frac{(b+a)(a-2b)}{a^2+ab+b^2} -1
= \frac{-2b(a+b)}{a^2+ab+b^2} = \frac{-2bc}{a^2+ab+b^2},
\end{align*}
we conclude $\deg G_{a,b,m}(T) = -a^2 = -\min(a,b)^2 = -N_{a,b,m}$.

\subsection{Case 3: $a \geq 2$ even and $b \geq 1$ odd}

While explicitly describing $G_{a,b,m}(T)$ by putting $X=R_{a,b,m}(T)$ and $Y=R_{-a,-b,m}(T)$ at $F_{a,b}(X,Y) = -P_{a}(-Y,-X) -P_{b}(-Y,-X) -P_{c}(-X,-Y)$, we only have to take extra care of $-P_{c}(-X,-Y)$, with others same as in Case 2.
Note that $G_{a,b}(T) = T^{-a^2} +T^{-b^2} -T^{-c^2}$. So,

\begin{align*}
G_{a,b,m}(T) = & -P_{a} \left( -\left( R_{a,b} \left( \frac{1}{T} \right) - \frac{1}{mT^m}E_{a,b} \left( \frac{1}{T} \right) \right),  -\left( R_{a,b}(T) + \frac{1}{mT^{m}}E_{a,b}(T) \right) \right) \\
& -P_{b} \left( -\left( R_{a,b} \left( \frac{1}{T} \right) - \frac{1}{mT^m}E_{a,b} \left( \frac{1}{T} \right) \right), -\left( R_{a,b}(T) + \frac{1}{mT^{m}}E_{a,b}(T) \right) \right) \\
& -P_{c} \left( -\left( R_{a,b}(T) + \frac{1}{mT^{m}}E_{a,b}(T) \right), -\left( R_{a,b} \left( \frac{1}{T} \right) - \frac{1}{mT^m}E_{a,b} \left( \frac{1}{T} \right) \right) \right) \\
= &  -P_{a} \left(-R_{a,b} \left( \frac{1}{T} \right), -R_{a,b}(T) \right)  -P_{b} \left( -R_{a,b} \left( \frac{1}{T} \right), -R_{a,b}(T) \right) \\
& -P_{c} \left( -R_{a,b}(T), -R_{a,b} \left( \frac{1}{T} \right) \right) \\
& + a \sum_{\substack{k,l \geq 0 \\ 0 \leq 2k + 3l \leq a}} (-1)^{k} \binom{k+l}{k} \binom{a-k-2l}{k+l} \frac{(-1)^{a-k-3l}}{a-k-2l} \sum_{\substack{0 \leq i \leq k \\ 0 \leq j \leq a-2k-3l \\ (i,j) \neq (0,0)}} A_{k,i}^{a-2k-3l,j} \\
& + b \sum_{\substack{k,l \geq 0 \\ 0 \leq 2k + 3l \leq b}} (-1)^{k} \binom{k+l}{k} \binom{b-k-2l}{k+l} \frac{(-1)^{b-k-3l}}{b-k-2l} \sum_{\substack{0 \leq i \leq k \\ 0 \leq j \leq b-2k-3l \\ (i,j) \neq (0,0)}} B_{k,i}^{b-2k-3l,j} \\
& + c \sum_{\substack{k,l \geq 0 \\ 0 \leq 2k + 3l \leq c}} (-1)^{k} \binom{k+l}{k} \binom{c-k-2l}{k+l} \frac{(-1)^{c-k-3l}}{c-k-2l} \sum_{\substack{0 \leq i \leq k \\ 0 \leq j \leq c-2k-3l \\ (i,j) \neq (0,0)}} C_{k,i}^{c-2k-3l,j}\\
= & \frac{1}{T^{a^2}} +\frac{1}{T^{b^2}} -\frac{1}{T^{c^2}} \\
& + a \sum_{\substack{k,l \geq 0 \\ 0 \leq 2k + 3l \leq a}} (-1)^{k} \binom{k+l}{k} \binom{a-k-2l}{k+l} \frac{(-1)^{a-k-3l}}{a-k-2l} \sum_{\substack{0 \leq i \leq k \\ 0 \leq j \leq a-2k-3l \\ (i,j) \neq (0,0)}} A_{k,i}^{a-2k-3l,j} \\
& + b \sum_{\substack{k,l \geq 0 \\ 0 \leq 2k + 3l \leq b}} (-1)^{k} \binom{k+l}{k} \binom{b-k-2l}{k+l} \frac{(-1)^{b-k-3l}}{b-k-2l} \sum_{\substack{0 \leq i \leq k \\ 0 \leq j \leq b-2k-3l \\ (i,j) \neq (0,0)}} B_{k,i}^{b-2k-3l,j} \\
& + c \sum_{\substack{k,l \geq 0 \\ 0 \leq 2k + 3l \leq c}} (-1)^{k} \binom{k+l}{k} \binom{c-k-2l}{k+l} \frac{(-1)^{c-k-3l}}{c-k-2l} \sum_{\substack{0 \leq i \leq k \\ 0 \leq j \leq c-2k-3l \\ (i,j) \neq (0,0)}} C_{k,i}^{c-2k-3l,j}.
\end{align*}
As before, we know
\begin{align}
\deg A_{k,i}^{a-2k-3l,j} = k \max(a,b) +(a-2k-3l)(a+b) -(i+j)(a^2+ab+b^2), \label{eq:26}\\
\deg B_{k,i}^{b-2k-3l,j} = k \max(a,b) +(b-2k-3l)(a+b) -(i+j)(a^2+ab+b^2), \label{eq:27} \\
\deg C_{k,i}^{c-2k-3l,j} = k(a+b) + (c-2k-3l)\max(a,b) -(i+j)(a^2+ab+b^2). \label{eq:28}
\end{align}

\subsubsection{$\max(a,b)=a$}
For $\max(a,b)=a$, we have
\begin{align*}
\text{The RHS of \eqref{eq:26}} & = ka +(a-2k-3l)(a+b) -(i+j)(a^2+ab+b^2) \\
& = -k(a+2b) -3(a+b)l + a(a+b) -(i+j)(a^2+ab+b^2), \\
\text{That of \eqref{eq:27}} & = ka +(b-2k-3l)(a+b) -(i+j)(a^2+ab+b^2) \\
& = -k(a+2b) -3(a+b)l +b(a+b) -(i+j)(a^2+ab+b^2), \\
\text{That of \eqref{eq:28}} & = k(a+b) + (c-2k-3l)a -(i+j)(a^2+ab+b^2) \\
& = k(-a+b) -3al +a(a+b) -(i+j)(a^2+ab+b^2).
\end{align*}
As before, \eqref{eq:26}, \eqref{eq:27}, and \eqref{eq:28} have maximum at $k=l=i=0,j=1$ with $\deg A_{0,0}^{a,1}=-b^2$, $\deg B_{0,0}^{b,1}=-a^2$, and $\deg C_{0,0}^{c,1}=-b^2$. So, $\deg G_{a,b,m}(T) = -b^2$ provided that $\alpha_{-b^2}(G_{a,b,m}) \neq 0$.
$\alpha_{-b^2}(G_{a,b,m})$ can be described as
\begin{align*}
& a \binom{0+0}{0} \binom{a-0-0}{0} \frac{(-1)^a}{a-0-0} \cdot \text{lc}(A_{0,0}^{a,1})
 + c \binom{0+0}{0} \binom{c-0-0}{0} \frac{(-1)^{c}}{c-0-0} \cdot \text{lc}(C_{0,0}^{c,1}) + 1 \\
& = \frac{ab}{a^2+ab+b^2} +\frac{b^2-a^2}{a^2+ab+b^2} +1 
= \frac{2b(a+b)}{a^2+ab+b^2} = \frac{2bc}{a^2+ab+b^2} \neq 0.
\end{align*}
So, we know $\alpha_{-b^2}(G_{a,b,m}) = \text{lc}(G_{a,b,m}) \neq 0$ and $\deg G_{a,b,m}(T) = -b^2 = - \min(a,b)^2 = -N_{a,b,m}$.

\subsubsection{$\max(a,b)=b$}
\begin{align*}
\text{The RHS of \eqref{eq:26}} & = kb +(a-2k-3l)(a+b) -(i+j)(a^2+ab+b^2) \\
& = -k(2a+b) -3(a+b)l + a(a+b) -(i+j)(a^2+ab+b^2), \\
\text{That of \eqref{eq:27}} & = kb +(b-2k-3l)(a+b) -(i+j)(a^2+ab+b^2) \\
& = -k(2a+b) -3(a+b)l +b(a+b) -(i+j)(a^2+ab+b^2), \\
\text{That of \eqref{eq:28}} & = k(a+b) + (c-2k-3l)b -(i+j)(a^2+ab+b^2) \\
& = k(a-b) -3bl +b(a+b) -(i+j)(a^2+ab+b^2).
\end{align*}
In this case, for $k=l=i=0,j=1$, $\deg A_{0,0}^{a,1}=-b^2$, $\deg B_{0,0}^{b,1}=-a^2$, and $\deg C_{0,0}^{c,1}=-a^2$. Since $\alpha_{-a^2}(G_{a,b,m})$ is
\begin{align*}
& b \binom{0+0}{0} \binom{b-0-0}{0} \frac{(-1)^b}{b-0-0} \cdot \text{lc}(B_{0,0}^{b,1}) \\
& + c \binom{0+0}{0} \binom{c-0-0}{0} \frac{(-1)^{c}}{c-0-0} \cdot \text{lc}(C_{0,0}^{c,1}) -1 \\
= & -\frac{b^2}{a^2+ab+b^2} -\frac{(b+a)(a-2b)}{a^2+ab+b^2} +1
= \frac{2b(a+b)}{a^2+ab+b^2} = \frac{2bc}{a^2+ab+b^2} \neq 0,
\end{align*}
we conclude $\deg G_{a,b,m}(T) = -a^2 = -\min(a,b)^2 = -N_{a,b,m}$.
Hence, the proof of Conjecture \ref{conj2} is complete.

\printbibliography





\end{document}